\documentclass[10pt,reqno]{amsart}

\usepackage[utf8]{inputenc}
\usepackage{comment}
\usepackage{amsmath}
\usepackage{mathtools}
\usepackage{graphicx}
\usepackage{amsfonts}
\usepackage{amsthm}
\usepackage{amssymb}
\usepackage{mathrsfs}	
\numberwithin{equation}{section}
\usepackage[all]{xy}
\usepackage{color}
\usepackage[left=3cm,right=3cm]{geometry}
\usepackage{bm} 

\usepackage{todonotes}

\DeclareMathOperator{\cc}{c}
\DeclareMathOperator{\ch}{ch}
\DeclareMathOperator{\Char}{Char}
\DeclareMathOperator{\dd}{\mathsf{d}}

\DeclareMathOperator{\GL}{GL}
\DeclareMathOperator{\gl}{{\mathfrak {gl}}}
\DeclareMathOperator{\Gr}{Gr}
\DeclareMathOperator{\oH}{H}

\DeclareMathOperator{\id}{id}
\DeclareMathOperator{\Log}{L}

\DeclareMathOperator{\rk}{rk}

\DeclareMathOperator{\ssl}{{\mathfrak sl}}
\DeclareMathOperator{\Sym}{Sym}
\DeclareMathOperator{\Rep}{Rep}
\DeclareMathOperator{\Tr}{Tr}
\newcommand{\ue}{\mathfrak{U}}

\newcommand{\mO}{\mathcal{O}}

\newcommand{\sS}{\mathsf{\mathbf{S}}}

\newcommand{\W}{\bigwedge}

\newcommand{\K}{\mathbb{C}}

\newcommand{\Nset}{\mathbb{N}}
\newcommand{\Qset}{\mathbb{Q}}
\newcommand{\Zset}{\mathbb{Z}}

\newtheorem{lemma}{Lemma}[section]
\newtheorem{proposition}[lemma]{Proposition}
\newtheorem{thm}[lemma]{Theorem}
\newtheorem{theorem}{Theorem}

\newtheorem{definition}[lemma]{Definition}
\newtheorem{question}[lemma]{Question}
\newtheorem*{questionno}{Question}

\theoremstyle{remark}
\newtheorem{example}[lemma]{Example}
\newtheorem{remark}[lemma]{Remark}

\newtheorem*{notation*}{Notation}

\title{Higher discriminants of vector bundles and Schur functors}

\author{Alessandro D'Andrea}

\author{Enrico Fatighenti}

\author{Claudio Onorati}
\address{\newline Alma Mater studiorum Universit\`a di Bologna\hfill\newline Dipartimento di Matematica\hfill\newline Piazza di Porta San Donato 5, 40126 Bologna, Italy} \email[A.~D'Andrea]{a.dandrea@unibo.it} \email[E.~Fatighenti]{enrico.fatighenti@unibo.it} \email[C.~Onorati]{claudio.onorati@unibo.it}
\date{}

\begin{document}

\begin{abstract}
    We prove some closed formulas for the logarithmic Chern character of a locally free sheaf. The argument used is representation-theoretic and we connect these formulas with the actions of some Casimir elements of $\ssl_r$. 

    As an application, we give a recipe to construct slope polystable modular bundles on hyper-K\"ahler manifolds from old ones.
\end{abstract}

\maketitle

\tableofcontents

\section*{Introduction}

If $X$ is a compact K\"ahler manifold and $E$ is a holomorphic vector bundle on $X$, it is in general an important and difficult problem the computation of the Chern classes $\cc_i(E)\in\oH^{2i}(X,\Zset)$. 

In a similar way, if $E$ is a vector bundle obtained from another vector bundle $E'$ via tensor operations, it is important and difficult to write Chern classes of $E$ in terms of Chern classes of $E'$. Examples are given by symmetric and exterior powers: in \cite{Dragutin} the author give a combinatorial expression of the Chern character of these operations. 
Extracting general formulas can be quite difficult even in the case of tensor product of two bundles (see \cite{Man16}).
\medskip

Our main result is the following explicit formula for the first three \emph{logarithmic Chern characters} of the Schur functor a vector bundle $E$. As a consequence of this result, we have explicit formulas for the first three Chern characters of Schur functors (see Appendix~\ref{appendix:Schur}). 

Let us recall the setting. If $E$ is a vector bundle of rank $r$, then we set
\begin{align*} 
    \Delta_1(E)= & \cc_1(E) \\
    \Delta_2(E)= & \cc_1(E)^2-2r\ch_2(E) \\
    \Delta_3(E)= & \cc_1(E)^3-3r\cc_1(E)\ch_2(E)+3r^2\ch_3(E)\,.
\end{align*}
These classes arise naturally from the logarithm expansion of the Chern character (see Section~\ref{section:log}). 

\begin{theorem}[Theorem~\ref{thm:fibrati}]\label{Theorem A}
    Let $X$ be a compact K\"ahler manifold or a proper smooth complex variety, and let $E$ be a vector bundle of rank $r$ on $X$. Denote by $\sS^\alpha E$ the Schur functor associated to a partition $\alpha=(\alpha_1,\dots,\alpha_r)$, and denote by $r_\alpha$ its rank. Then
    \begin{align*}
        \Delta_1(\sS^\alpha E) = & \, |\alpha|\frac{r_\alpha}{r} \Delta_1(E) \\
        \Delta_2(\sS^\alpha E) = & \, \frac{\dot\delta_r^{(2)}(\alpha)}{(r-1) (r+1)} \left(\frac{r_\alpha}{r}\right)^2 \Delta_2(E) \\
        \Delta_3(\sS^\alpha E) = & \, \frac{\dot\delta_r^{(3)}(\alpha)}{(r-2) (r-1) (r+1) (r+2)} \left(\frac{r_\alpha}{r}\right)^3 \Delta_2(E)\,,
    \end{align*}
    where $\dot\delta_r^{(2)}(\alpha)$ and $\dot\delta_r^{(3)}(\alpha)$ are polynomials in $\alpha=(\alpha_1, \ldots, \alpha_r)$ as introduced in Theorem~\ref{thm:casimiri}.
\end{theorem}
To the best of our knowledge, this is the first time explicit formulas for Chern classes of Schur functors are exhibited in this generality. Notice also that from the classes $\Delta_1(\sS^\alpha E)$, $\Delta_2(\sS^\alpha E)$ and $\Delta_3(\sS^\alpha E)$ one can recursively also find the classes $c_1(\sS^\alpha E)$, $\ch_2(\sS^\alpha E)$ and $\ch_3(\sS^\alpha E)$. For book-keeping reasons, we collect in Section~\ref{appendix:Schur} these expressions.

The polynomials $\dot\delta_r^{(2)}(\alpha)$ and $\dot\delta_r^{(3)}(\alpha)$ are explicit. In fact we have a purely algebraic interpretation of $\dot\delta_r^{(k)}$ as characters of two explicit Casimir elements of $\ssl_r$ (see Section~\ref{section:Casimir actions}). 

The proof of Theorem~\ref{Theorem A} is representation-theoretic. Using the splitting principle, we can reduce to work formally with a complex vector space $V$ of dimension $r$ and its Schur modules $\sS^\alpha V$, i.e.\ the corresponding irreducible polynomial representations of $\GL(V)$. We review in Section~\ref{section:polynomial reps} some basic definitions and facts about polynomial representations. 
The two main results in this setting are Theorem~\ref{thm:Delta 2 e 3} and Theorem~\ref{thm:casimiri}, proved respectively in Section~\ref{section:discriminants} and Section~\ref{section:Casimir actions}. The combination of the two gives Theorem~\ref{Theorem A}.

Extending these results to degrees higher than $3$ seems to be difficult. As we point out in Remark~\ref{rmk:più di tre non si va}, one can hope to achieve such a generalisation by rigidifying the situation, i.e.\ fixing some numerical invariants. In Section~\ref{section:higher discriminants?} we discuss a bit more about this higher degree situation, leaving two open questions.
\medskip

Beside being an interesting computational and combinatorial problem, there are other geometric motivations for considering (higher) discriminants. In fact we wish to point out that they are strictly connected to the notion of stability. For example, the class $\Delta_2$ is the usual discriminant, and the famous Bogomolov Theorem states that if a vector bundle $E$ is $\omega$-polystable, then $\Delta_2(E).\omega^{\dim X-2}\geq0$. (Here $\omega$ is a K\"ahler class.) Other connections can be found in \cite{DrezetLog,Ful20}.

From our point of view, another important motivation to consider this problem comes from the geometry of compact hyper-K\"ahler manifolds. Recently O'Grady have introduced the notion of \emph{modular} vector bundles on compact hyper-K\"ahler manifolds as an attempt to generalise to higher dimensions the theory of moduli spaces of bundles on K3 surfaces (see \cite{O'Grady:Modular}). In this context, a vector bundle $E$ is called modular if $\Delta_2(E)$ belongs to a line in $\oH^4(X,\Qset)$ intrinsically attached to $X$ (see Definition~\ref{defn:modular}).

Because of the numerical nature of this definition, it seems to be particularly difficult to construct examples of modular vector bundles. Recent works in this direction include \cite{Bottini,Fa24, FO24, FT25, O'Grady:Modular with many moduli}.

A straightforward corollary of Theorem~\ref{Theorem A} yields the following result, which allows us to construct infinitely many polystable modular bundles starting from old ones.

\begin{theorem}[Theorem~\ref{thm:modular}]
    Let $X$ be a compact hyper-K\"ahler manifold. If $E$ is a slope polystable modular vector bundle, then for every $\alpha$ the Schur functor $\sS^\alpha E$ is a slope polystable modular vector bundle.
\end{theorem}

Even if $E$ is slope stable, in general $\sS^\alpha E$ would not be stable (cf.\ Remark~\ref{remark:not stable}). 
\medskip

We conclude the article with two appendices. In the first one, Appendix~\ref{appendix:low ranks}, we comment on an interesting property of discriminants that stops to work again from degree $4$ on-wards. This gives us the opportunity to explore higher discriminants from a different point of view. In Appendix~\ref{appendix} we list useful formulas for the Chern character (in degree up to three) of symmetric and exterior powers bundles, and for Schur bundles, which are obtained as a consequence of Theorem~\ref{Theorem A}. 

\subsection*{Acknowledgments}
This research has been partially funded by the European Union - NextGenerationEU under the
National Recovery and Resilience Plan (PNRR) - Mission 4 Education and research - Component 2
From research to business - Investment 1.1 Notice Prin 2022 - DD N. 104 del 2/2/2022, from title
“Symplectic varieties: their interplay with Fano manifolds and derived categories”, proposal code
2022PEKYBJ – CUP J53D23003840006.
The second and third authors are members of the INDAM-GNSAGA group.

\section{Review of polynomial representations of $\GL_r$}\label{section:polynomial reps}

Let $V$ be a complex vector space of dimension $r$ and $\GL_r=\GL(V)$ the corresponding algebraic group. 
\medskip

Choosing a basis of $V$ provides a choice of a maximal sub-torus $T \subset \GL_r$ of diagonal matrices. Every polynomial group homomorphism $T \to \K^*$ is of the form
$$ \lambda\colon \operatorname{diag}(t_1, \dots, t_r) \mapsto t_1^{\lambda_1} \dots t_r^{\lambda_r}, $$
so that the set $\widehat T$ of all polynomial group homomorphisms $T \to \K^*$ can be identified with $\Nset^r$ via $\lambda \mapsto (\lambda_1, \dots, \lambda_r)$.
Restricting the action of $\GL_r$ on a finite-dimensional polynomial representation to the torus $T$ provides
a direct sum decomposition
$$W = \bigoplus_{\lambda \in \widehat{T}} W(\lambda),$$
where $W(\lambda) = \{v \in W \mid t.v = \lambda(t) v\}$ is the {\em weight space} corresponding to $\lambda$. 
\medskip

Polynomial actions of $\GL_r$ are completely reducible and all irreducible polynomial representations, called {\em Schur representations}, arise as direct summands of some $V^{\otimes d}$. 
Recall that Schur representations are indexed by partitions $\alpha=(\alpha_1,\dots,\alpha_r)$ and denoted by $\sS^\alpha V$. The size of $\alpha$ is $|\alpha|=\alpha_1+\dots+\alpha_r$ and if $\sS^\alpha V\subset V^{\otimes d}$, then $|\alpha|=d$ and we will say that $\sS^\alpha V$ is a Schur representation of degree $d$.
\medskip

With each polynomial representation $W$ of $\GL_r$, one may associate its character
$$\Char(W) = \sum_{\lambda\in \Nset^r} \dim W(\lambda)\, q^\lambda \in \Zset[q_1, \dots, q_r],$$
where $q^\lambda = q_1^{\lambda_1} \cdots q_r^{\lambda r}$. 
It is immediate to see that
$$\Char(W_1 \oplus W_2) = \Char(W_1) + \Char(W_2) \quad \mbox{and} \quad \Char(W_1 \otimes W_2) = \Char(W_1) \cdot \Char(W_2).$$

The character of a polynomial representation is always a symmetric polynomial in $q_1, \dots, q_r$, and the character of a Schur representation of degree $d$ is homogeneous of degree $d$. 

\begin{example}
    Among Schur functors we have the symmetric powers $S^k V$ and the exterior powers $\W^k V$. This defines three famous sets of symmetric polynomials:
    \begin{itemize}
        \item elementary symmetric polynomials $\sigma_k=\Char(\W^k V)$;
        \item complete symmetric polynomials $h_k=\Char(S^k V)$;
        \item Schur polynomials $s_\alpha=\Char(\sS^\alpha V)$.
    \end{itemize}
\end{example}

Finally, let us also recall that every action of $\GL_r$ induces corresponding representations of the Lie algebra $\gl_r$ and of its universal enveloping algebra $\ue(\gl_r)$, which has a canonical cocommutative Hopf algebra structure where elements of $\gl_r$ are all primitive. This will be used in Section~\ref{section:Casimir actions} to study some special characters.

\subsection{The Grothendieck group of representations}

Let $\Rep=\Rep(\GL_r)$ be the Grothendieck ring of polynomial representations of $\GL_r$, i.e.\ the abelian group of formal $\Zset$-linear combinations of isomorphism classes of finite-dimensional polynomial representations of $\GL_r$ endowed with the unique multiplication extending tensor product of representations. We also consider the $\Qset$-linear extension of $\Rep$, i.e.\ the $\Qset$-algebra $\Rep_\Qset = \Rep \otimes_\Zset \Qset$. 
(Non-zero) Schur representations provide a free basis of both algebras.
\medskip

One may extend $\Qset$-linearly $\Char$ to a $\Qset$-algebra homomorphism
\begin{equation}\label{eqn:Char}
\Char\colon \Rep_\Qset \to \Qset[q_1, \dots, q_r]
\end{equation}
which indeed defines an isomorphism between $\Rep_\Qset$ and the sub-algebra of symmetric polynomials. 

Moreover $\Char(U)$ is a homogeneous polynomial of degree $d$ if and only if $U$ is a $\Qset$-linear combination of Schur representations of degree $d$. We thus endow $\Rep_\Qset$ with the grading corresponding to the degree, i.e.\
$$\Rep_\Qset = \bigoplus_{d\geq 0} \Rep_\Qset^d, $$
thus making $\Char$ into a homogeneous homomorphism.
\medskip

It is a standard fact that elementary symmetric polynomials $\sigma_1, \dots, \sigma_r$ are $\Zset$-algebra generators of the ring $\Zset[q_1, \dots, q_r]^{\mathfrak{S}_r}$ of symmetric polynomials with integral coefficients, so that wedge powers $\wedge^k V$ for $k = 1, \dots, r$ generate $\Rep$ over $\Zset$ (hence also $\Rep_\Qset$ over $\Qset$). The same holds for symmetric powers $S^k V$ for $k = 1, \dots, r$, as  complete homogeneous symmetric polynomials $h_1, \dots, h_r$ are also known to generate $\Zset[q_1, \dots, q_r]^{\mathfrak{S}_r}$. In particular, the symmetric power representations $S^1 V,\dots,S^r V$ are $\Zset$-algebra generators of $\Rep$, and similarly $\Qset$-algebra generators of $\Rep_\Qset$.

\begin{remark}
Each $S^k V$ is a representation of degree $k$. The homogeneous component $\Rep^d$ is generated by those monomials of symmetric power representations which have degree $d$.
To introduce a notation we will use later, with any $\alpha=(\alpha_1,\dots,\alpha_t)\in \Nset^t$ we associate the representation $\Sym^\alpha V:=S^{\alpha_1} V\cdots S^{\alpha_t}V\in \Rep$. Then $\Rep^d$ is $\Zset$-linearly generated by all $\Sym^{\alpha}V$ such that $|\alpha|=d$.
\end{remark}
 
\begin{remark}\label{rmk:P k}
Power sums polynomials 
\begin{equation}\label{eqn:p k}
p_k = q_1^k + \dots + q_r^k, 
\end{equation}
for $k = 1, \dots, r$, do not generate $\Zset[\sigma_1, \dots, \sigma_r]$ but are $\Qset$-algebra generators of $\Qset[\sigma_1, \dots, \sigma_r]$. If we define $P_k:=\Char^{-1}(p_k)\in\Rep_\Qset$, then the $P_k$, for $k=1,\dots,r$ generate $\Rep_\Qset$ as a $\Qset$-algebra. As before, if we set $P_\alpha := P_{\alpha_1} \cdot \ldots \cdot P_{\alpha_t}$ then $\Rep^d$ is $\Qset$-linearly generated by all $P_\alpha$ such that $|\alpha| = d$.

For instance, $P_1=V$ and $P_2 = S^2 V - \wedge^2 V$. In general $P_k$ is not an honest representation of $\GL_r$: it is known that $P_n$ is the alternating sum of all Schur {\em hook representations}\footnote{i.e. those corresponding to a hook partition $(k, 1, 1, \dots, 1, 0, \dots, 0)$.} of degree $n$, but we will not need this fact. Notice that $P_0$ equals $r$ times the trivial one-dimensional representation.
\end{remark}

\subsection{Chern character of a representation}

In analogy with bundles, we define now the Chern character $\ch(W)$ of a polynomial representation $W$.

First of all, denote by
\[ e\colon \Qset[q_1,\dots,q_r]\longrightarrow\Qset[[a_1,\dots,a_r]] \]
the evaluation homomorphism obtained by setting $q_i\mapsto e^{a_i}\in \Qset[[a_i]] \subset \Qset[[a_1, \dots, a_r]]$. The indeterminates $a_1,\dots,a_r$ are called \emph{Chern roots} (of $V$).

\begin{definition}\label{defn:ch for rep}
    The \emph{Chern character} morphism is
    \[ \ch:=e\circ\Char\colon  \Rep_\Qset \to \Qset[[a_1, \dots, a_r]]. \]
\end{definition}

Given $U\in\Rep_\Qset$, we group terms in $\ch(U)$ by degree, so that
$$\ch(U) = \sum_{k\geq 0} \ch_k(U),$$
where $\ch_k(U)$ is a homogeneous symmetric polynomial in $a_1, \dots, a_r$ of degree $k$. 

\begin{example}\label{example:homogeneity}
Let $P_d$ be the elements defined in Remark \ref{rmk:P k}. By definition, we have 
\[ \ch_k(P_d) = \frac{(d a_1)^k + \dots + (d a_r)^k}{k!}= d^k \frac{p_k(a_1, \dots, a_r)}{k!} . \] 
Notice that in particular $\ch_0(P_d) = r$ for every $d\geq 1$.
As a consequence, 
\begin{equation} 
\ch_k(P_d) = d^k \ch_k(P_1). 
\end{equation}
More in general, if $\mathsf{c} = \ch_{k_1} \dots \ch_{k_s}$, then
$$\mathsf{c} (P_d) = d^{k_1 + \dots + k_s} \mathsf{c}(P_1).$$
\end{example}

\begin{example}[Svrtan]\label{example:Svrtan}
In \cite{Dragutin}, closed formulas are given for $\ch(S^m V)$ and $\ch(\W^m V)$ (see Sections~\ref{appendix:symmetric} and~\ref{appendix:exterior}). The statements of such formulas are in terms of Chern characters of vector bundles, but the proofs are purely combinatorial and based on the splitting principle, so that they are indeed set in this algebraic framework.

As an example, let us state here the formula for the symmetric power (\cite[Theorem~4.8]{Dragutin}):
\begin{equation}\label{eqn:ch of Sym}
\ch(S^m V)=\sum_\alpha\frac{1}{\lVert\alpha\rVert}\sum_{\beta\leq\alpha}\binom{m+r-1}{m-|\beta|}\overleftarrow{\beta}!S(\alpha,\beta)\ch_\alpha(V). 
\end{equation}
The first summation is taken over all partitions $\alpha=(\alpha_1,\dots,\alpha_t)$ such that $\alpha_i\neq0$ for every $i=1,\dots,t$, and $\ch_\alpha(\tilde{V})$ stands for the product $\prod_i\ch_{\alpha_i}(V)$. The norm $\lVert\alpha\rVert$ is the product $a_1!a_2!\cdots a_t!$, where $a_i=\sharp\{j\mid \alpha_j=i\}$. The second summation is taken over all $\beta\in\Nset^s$ such that $\beta\leq\alpha$ and $\ell(\beta)=\ell(\alpha)$ (in particular also $\beta_i\neq0$ for every $i$). The symbol $\overleftarrow{\beta}$ stands for the partition $(\beta_1-1,\dots,\beta_t-1)$ and $\overleftarrow{\beta}!=\prod_i(\beta_i-1)!$.  
Finally, $S(\alpha,\beta)=\prod_iS(\alpha_i,\beta_i)$, where $S(x,y)$ is the Stirling number of the second kind.
\end{example}

\begin{remark}\label{rmk:ch_0=rk}
    If $U$ is a representation of $\GL_r$, then 
    \[ \ch_0(U)=\dim U. \]
    The same follows when $U\in\Rep_\Qset$ (alternatively one can define the dimension of $U$ in this way).
\end{remark}

Finally, let us remark that for any $U_1,U_2\in\Rep_\Qset$ we have
\begin{equation}\label{eqn:ch somma e prodotto}
    \ch(U_1 + U_2)=\ch(U_1)+\ch(U_2) \qquad\mbox{and}\qquad \ch(U_1\cdot U_2)=\ch(U_1)\ch(U_2).
\end{equation}

\section{Discriminants}\label{section:discriminants}

\subsection{Slope}\label{section:slope}

From (\ref{eqn:ch somma e prodotto}) we obtain additivity
$$\ch_1(U_1 + U_2) = \ch_1(U_1) + \ch_1(U_2)$$ and the equality
$$\ch_1(U_1 \cdot U_2) = \ch_0(U_2) \ch_1(U_1) + \ch_0(U_1) \ch_1(U_2),$$
that we call {\em log-multiplicativity} as it may be rephrased as
\begin{equation}\label{eqn:slope additive}
\frac{\ch_1}{\ch_0}(U_1\cdot U_2) = \frac{\ch_1}{\ch_0}(U_1) + \frac{\ch_1}{\ch_0}(U_2).
\end{equation}

\begin{proposition}\label{prop:slope}
If $U \in \Rep_\Qset^d$, then
$$\frac{\ch_1}{\ch_0}(U) = d \, \frac{\ch_1}{\ch_0}(V).$$
\end{proposition}
\begin{proof}
By Example \ref{example:homogeneity}, we already know that $\ch_1(P_d) = d \ch_1(P_1)$ and $\ch_0(P_d) = r = \ch_0(P_1)$, thus showing the claim when $U = P_d$ (recall that $P_1=V$). 

If now $\alpha=(\alpha_1,\dots,\alpha_t)$ with $|\alpha|=d$, then by equality (\ref{eqn:slope additive}) it also follows that 
\begin{equation}\label{eqn:P alpha} 
\frac{\ch_1}{\ch_0}(P_\alpha)=d\,\frac{\ch_1}{\ch_0}(P_1),
\end{equation}
where $P_\alpha=P_{\alpha_1}\cdots P_{\alpha_t}$.
The statement then follows from the fact that the $P_\alpha$'s, for $|\alpha| = d$, are $\Qset$-linear generators of $\Rep_\Qset^d$. 
\end{proof}

The ratio $\frac{\ch_1}{\ch_0}$ is known as {\em slope}, and Proposition~\ref{prop:slope} shows in particular that the slope of each Schur representation $S^\alpha V$ is $|\alpha|$ times the slope of $V$. Therefore we see that $\Rep$ is (essentially) graded by the slope.

\subsection{Logarithmic Chern character}\label{section:log}

What makes Proposition~\ref{prop:slope} to work is the logarithmic equality (\ref{eqn:slope additive}). In order to extend it to higher degrees, let us then introduce the logarithmic Chern character.

\begin{definition}\label{defn:Log chern}
    For any $U\in\Rep_\Qset$, define the \emph{logarithmic Chern character} as
    \[ \Log_+(U) = \log \frac{\ch(U)}{\ch_0(U)} = \log\left(1 + \frac{\ch_1(U)}{\ch_0(U)} + \frac{\ch_2(U)}{\ch_0(U)} + \cdots\right) = \sum_{k>0} (-1)^{k+1} \frac{\Delta_k(U)}{k \ch_0(U)^k}. \]
\end{definition}
As an example, let us expand the first values of $\Delta_k$:
\begin{align}\label{eqn:Delta 1 2 3}
\begin{split}
    \Delta_1(U)= & \ch_1(U) \\
    \Delta_2(U)= & \ch_1(U)^2-2\ch_0(U)\ch_2(U) \\
    \Delta_3(U)= & \ch_1(U)^3-3\ch_0(U)\ch_1(U)\ch_2(U)+3\ch_0(U)^2\ch_3(U)\,.
\end{split}
\end{align}

The multiplicativity of $\ch$ translates into \emph{log-multiplicativity} of $\Log_+$,
\begin{equation}\label{eqn:L log-multiplicative}
    \Log_+(U_1\cdot U_2) = \Log_+(U_1) + \Log_+(U_2)\,.
\end{equation}

Before continuing towards the main result of the section, let us state another useful definition.

\begin{definition} \label{def:dk}
    For any $U\in\Rep_\Qset$, define
    \[ \dd_k = \frac{\Delta_k(U)}{k\ch_0(U)^{k-1}}. \]
\end{definition}

Equality (\ref{eqn:L log-multiplicative}) translates in 
\begin{equation}\label{eqn:log multiplicativity}
    \frac{\dd_k}{\ch_0}(U_1 \cdot U_2) = \frac{\dd_k}{\ch_0}(U_1) + \frac{\dd_k}{\ch_0}(U_2)
\end{equation}
or, equivalently, in
\begin{equation}\label{eqn:log multiplicativity dk}
\dd_k(U_1\cdot U_2)=\ch_0(U_2)\dd_k(U_1)+\ch_0(U_1)\dd_k(U_2).
\end{equation}
\smallskip

Moreover, $\dd_1(U)=\ch_1(U)$ is linear with respect to the sum operation. On the other hand, it is easy to see that in general $\dd_k$ is not additive for $k\geq2$ (see Example~\ref{example:counterexample to additivity}).

However, a \emph{weak additivity} holds for $\dd_2, \dd_3$.

\begin{proposition}\label{prop:weak additivity}
Fix $d\geq0$. If $U_1,U_2\in\Rep_\Qset^d$, then 
\[ \dd_k(U_1 + U_2)=\dd_k(U_1) + \dd_k(U_2) \]
for $k=2,3$.
\end{proposition}
\begin{proof}
As $$\Log_+ = \sum_{k>0} \frac{\dd_k}{\ch_0}$$
satisfies 
$\ch_0 \cdot \exp(\Log_+) = \ch$, then each coefficient of 
\begin{equation}\label{eqn:exp di log}
\ch_0 \cdot \exp(\Log_+) = \ch_0\cdot \left(1 + \frac{\dd_1}{\ch_0} + \left(\frac{\dd_2}{\ch_0} + \frac{\dd_1^2}{2\ch_0^2}\right) + \left(\frac{\dd_3}{\ch_0} + \frac{\dd_1 \dd_2}{\ch_0^2} + \frac{\dd_1^3}{6\ch_0^3}\right) + \dots \right)
\end{equation}
must be additive. In particular, both
$$\dd_2 + \frac{1}{2}\cdot \frac{\dd_1}{\ch_0}\cdot \dd_1 \qquad \mbox{and} \qquad \dd_3 + \frac{\dd_1}{\ch_0} \cdot \dd_2 + \left(\frac{\dd_1}{\ch_0}\right)^2 \cdot \frac{\dd_1}{6}$$
are additive on $\Rep_\Qset$.

However, when restricting to $\Rep_\Qset^d$, the first expression equals $\dd_2 + d/2 \cdot \dd_1$ (see Proposition~\ref{prop:slope}); as $\dd_1=\ch_1$ is additive, then also $\dd_2$ must be additive. 

Similarly, on $\Rep^d_\Qset$ the second expression equals $\dd_3 + d \cdot \dd_2 + d^2/6 \cdot \dd_1$ and we already know that $\dd_1$ and $\dd_2$ are additive on $\Rep_\Qset^d$. Thus $\dd_3$ must also be additive and the proof is concluded.
\end{proof}

Recall that additivity on a rational vector space, such as $\Rep_\Qset^d$, is equivalent to $\Qset$-linearity.

\begin{remark}\label{rmk:più di tre non si va}
    Proposition~\ref{prop:weak additivity} stops to work from degree $4$ on: it is not difficult to find counter-examples. We digress a bit more about higher discriminants in Section~\ref{section:higher discriminants?}.
\end{remark}

\begin{example}\label{example:counterexample to additivity}
    Proposition~\ref{prop:weak additivity} is sharp: it is easy to find a counter-example to additivity. Let us take $r=2$ for simplicity. Then by Example~\ref{example:Svrtan} (see also Section~\ref{appendix:symmetric}) we have that
    \[ \ch(S^2 V)=(3, 3\ch_1(V), \frac{1}{2}\ch_1(V)^2+4\ch_2(V),\cdots ), \]
    so that in particular $\dd_2(S^2 V)=\dd_k(V)$.

    Therefore, by additivity of Chern character, we have
    \[ \ch(V\oplus S^2V)=(5,4\ch_1(V),\frac{1}{2}\ch_1(V)^2+5\ch_2(V),\cdots). \]
    By an explicit computation,
    \[ \dd_2(V\oplus S^2V)=5\ch_2(V)-\frac{11}{10}\ch_1(V)^2\neq 8\ch_2(V)-2\ch_1(V)^2=\dd_2(V)+\dd_2(S^2V). \]
\end{example}

Our main result is the following generalisation of Proposition~\ref{prop:slope}.

\begin{thm}\label{thm:Delta 2 e 3}
If $U\in\Rep^d_\Qset$, then for $k=2,3$
$$\dd_k(U) = f^{(k)}_U(r) \dd_k(V) $$
for opportunely chosen polynomials $f^{(k)}_U\in \Qset[x]$ of degree at most $d-1$.
\end{thm}
\begin{proof}
By Example~\ref{example:homogeneity}, it follows that $\ch_k(P_\ell)=\ell^k\ch(P_1)$ and more generally that $\mathsf{c}(P_\ell)=\ell^{k_1+\cdots+k_t}\mathsf{c}(P_1)$, where $\mathsf{c}=\ch_{k_1}\cdots\ch_{k_t}$. In particular we have that 
\[ \frac{\dd_k}{\ch_0}(P_\ell) = \ell^k \frac{\dd_k}{\ch_0}(V) \] 
for all $k$ and all $\ell\geq1$. 

As usual, if we write $P_\alpha=P_{\alpha_1}\cdots P_{\alpha_t}$, then by log-multiplicativity it follows that 
\[ \frac{\dd_k}{\ch_0}(P_\alpha)=\left(\sum_i \alpha_i^k\right)\frac{\dd_k}{\ch_0}(P_1). \]
In particular $\dd_k(P_\alpha)$ is a scalar multiple of $\frac{\ch_0(P_\alpha)}{\ch_0(V)} \dd_k(V) = r^{t-1} \dd_k(V)$. 

Finally, let us set $k=2$ or $k=3$. For fixed degree $d$ the monomials $P_\alpha$ generate $\Rep_\Qset^d$ and then we conclude by weak additivity (see Proposition~\ref{prop:weak additivity}), which implies $\Qset$-linearity of the restriction of $d_k$ to $\Rep_\Qset^d$.
\end{proof}

\begin{example}
We have $\frac{\dd_2}{\ch_0}(P_2) = 4 \frac{\dd_2}{\ch_0}(V)$. As $\ch_0(V) = \ch_0(P_2)$, we obtain $\dd_2(P_2) = 4 \dd_2(V)$. Also, $\frac{\dd_2}{\ch_0}(V^{\otimes 2}) = 2 \frac{\dd_2}{\ch_0}(V)$. Here, $\ch_0(V^{\otimes 2}) = r^2$ so that $\dd_2(V^{\otimes 2}) = 2r \dd_2(V)$.

Now, $V^{\otimes 2} = S^2 V \oplus \wedge^2 V$ and $P_2 = S^2 V- \wedge^2 V$ so that 
\[ S^2 V = \frac{V^{\otimes 2} + P_2}{2}\qquad \mbox{and} \qquad \wedge^2 V = \frac{V^{\otimes 2} - P_2}{2}. \] 
We conclude by weak additivity that
$$\dd_2(S^2 V) = (r + 2) \dd_2(V) \qquad \mbox{and} \qquad \dd_2(\wedge^2 V) = (r - 2) \dd_2(V),$$
and equivalently that
$$\frac{\dd_2}{\ch_0}(S^2 V) = 2\frac{r + 2}{r+1} \cdot \frac{\dd_2}{\ch_0}(V) \qquad\mbox{and}\qquad \frac{\dd_2}{\ch_0}(\wedge^2 V) = 2\frac{r - 2}{r-1} \cdot \frac{\dd_2}{\ch_0}(V).$$

Also 
$$\Delta_2(S^2 V) = \frac{(r+1)(r + 2)}{2} \Delta_2(V) \qquad \mbox{and} \qquad \Delta_2(\wedge^2 V) = \frac{(r-1)(r - 2)}{2} \Delta_2(V).$$
\end{example}

\begin{example}\label{example:symmetric}
    Using formula (\ref{eqn:ch of Sym}) in Example~\ref{example:Svrtan} (see Section~\ref{appendix:symmetric}), we can explicitly compute $f^{(k)}_{S^m V}(r)$ in Theorem~\ref{thm:Delta 2 e 3}. The result is the following:
    \begin{equation}\label{eqn:Delta 2 Sym}
        f^{(2)}_{S^m V}(r)=\frac{m(m+r)}{(r+1)}\frac{r_m}{r}
    \end{equation}
    and
    \begin{equation}\label{eqn:Delta 3 Sym}
        f^{(3)}_{S^m V}(r)=\frac{m(m+r)(2m+r)}{(r+1)(r+2)}\frac{r_m}{r},
    \end{equation}
    where in both cases $r_m=\binom{m+r-1}{r-1}$ denotes the dimension of $S^m V$.
\end{example}

We conclude with the following easy remark about the additivity properties of the coefficients $f^{(k)}_U(r)$, which will be useful later.
\medskip

\begin{proposition}\label{prop:properties of fU}
Set $k=2,3$.
\begin{enumerate}
    \item If $U_1,U_2\in\Rep_\Qset$, then
    \[ f^{(k)}_{U_1\cdot U_2}(r)=\ch_0(U_2)f^{(k)}_{U_1}(r)+\ch_0(U_1)f^{(k)}_{U_2}(r). \]
    \item If $U_1,U_2\in\Rep_\Qset^d$, then 
    \[ f^{(k)}_{U_1 + U_2}(r)=f^{(k)}_{U_1}(r)+f^{(k)}_{U_2}(r). \]
\end{enumerate}
\end{proposition}
\begin{proof}
    The first equality follows from log-multiplicativity (\ref{eqn:log multiplicativity dk}); the second equality follows from weak additivity (Proposition~\ref{prop:weak additivity}).
\end{proof}

\subsection{What about higher discriminants?}\label{section:higher discriminants?}

As we already noticed in Remark~\ref{rmk:più di tre non si va}, the results in this sections seem not to generalise to higher discriminants.

There are two possible solutions of this problem. The first one consists in "ignoring" the problem: for a representation $W$, we only look at the coeffcient of $\dd_k(W)$ along the coordinate $\dd_k(V)$. This point of view seems to behave well and should lead to an algorithm to compute plethysms. We will explore and study this phenomenon in another work.

The second possible solution consists in "perturbing" the logarithmic classes in order to get the sought result. Let us explain this approach by having a look at the case of degree $4$. Expanding the expression in Definition~\ref{defn:Log chern}, we find
\[
\Delta_4(W)= \ch_1(W)^4 - 4\ch_0(W) \ch_1(W)^2 \ch_2(W) + 2 \ch_0(W)^2 (\ch_2(W)^2 + 2 \ch_1(W) \ch_3(W)) - 4 \ch_0(W)^3 \ch_4(W).
\]
One can check that indeed $\Delta_4(S^m V)$ is not a multiple of $\Delta_4(V)$.

\begin{questionno}
    Can we modify $\Delta_4$ in order to get a statement like Theorem~\ref{thm:Delta 2 e 3} for symmetric representations?
\end{questionno}

The answer to the question above is surprisingly yes. For any $t\geq1$, define
\begin{align*} 
\Delta_{4,t}(W)= & \, t \ch_1(W)^4-4t\ch_0(W)\ch_1(W)^2\ch_2(W) \\ 
+ & \,2\ch_0(W)^2\left((t+1)\ch_2(W)^2+2(t-1)\ch_1(W)\ch_3(W)\right) \\
- & \,4(t-1)\ch_0(W)^3\ch_4(W).
\end{align*}

\begin{lemma}\label{lemm:Delta 4 t for Sym}
    For any $m\geq0$ we have 
    \[ \Delta_{4,r}(S^m V)=f_{S^m V}^{(4)}(r)\left(\frac{r_m}{r}\right)^4\Delta_{4,r}, \]
    where 
    \[ f_{S^m V}^{(4)}(r)=\frac{m(m+r)(m^2+rm+r(r+1))}{(r+1)(r+2)(r+3)}. \]
\end{lemma}
\begin{proof}
    This is a direct computation, using formula (\ref{eqn:ch of Sym}).
\end{proof}

\begin{remark}
    So far we have worked with a vector space of fixed dimension $r$. The classes $\Delta_k$ have the same shape independently of $r$. Notice instead that in Lemma~\ref{lemm:Delta 4 t for Sym} we have different classes for different $r$.
\end{remark}

The class $\Delta_{4,t}$ has been obtained by brute force. We wrote down a parametric expression of the tentative $\Delta_{4,t}$ and then we imposed the equality in Lemma~\ref{lemm:Delta 4 t for Sym}: in this way we obtained a unique solution (the one displayed). 

The same is also true for $\Delta_2$ and $\Delta_3$: those are the unique classes for which the symmetric representations transforms as in Example~\ref{example:symmetric}. 

\begin{question}
    For any $k\geq5$, can we find modifications $\Delta_{k,t}$ of $\Delta_k$ such that $\Delta_{k,r}(S^m V)$ is a multiple of $\Delta_{k,r}(V)$?
\end{question}

It is easy to see that, for example, if $\alpha=(m,1,0,\dots,0)$ then $\Delta_{4,t}(\sS^\alpha V)$ is not a multiple of $\Delta_{4,t}$. This can be seen as a lack of weak additivity in degrees higher than 3 (even in this modified form). Nevertheless, numerical examples seem to suggest that modifications exist for partitions other than symmetric ones. 

In the following, we denote by $\delta_1=(1,0,\dots,0)$ the first fundamental weight.

\begin{question}
    For an adequately chosen partition $\alpha$, does there exist a modification $\Delta_{4,\alpha}$ such that $\Delta_{4,\alpha}(\sS^{\alpha+m\delta_1}V)$ is a multiple of $\Delta_{4,\alpha}(\sS^\alpha V)$?
\end{question}

\section{Casimir actions}\label{section:Casimir actions}

In this section we determine the expressions $f^{(k)}_U(r)$ appearing in Theorem~\ref{thm:Delta 2 e 3} when $U=\sS^\alpha V$ is a Schur functor. The key observation is to identify them with the action of quadratic and cubic Casimir elements of the \emph{universal enveloping algebra} $\mathfrak {U}(\gl(n))$. 

Let us state here the main result of the section. Recall that $\alpha=(\alpha_1,\dots,\alpha_r)$ is a partition and $r_\alpha$ denotes the dimension of the Schur module $\sS^\alpha V$.

\begin{thm}\label{thm:casimiri}
    Keep notations as in Theorem~\ref{thm:Delta 2 e 3}. 
    \begin{enumerate}
        \item For the operator $\dd_2$ we have
        \[ f^{(2)}_{\sS^\alpha V}(r)=\frac{\dot\delta_r^{(2)}(\alpha)}{(r-1)(r+1)}\frac{r_\alpha}{r}, \]
        where 
        \begin{equation*}
        \dot\delta_r^{(2)}(\alpha)=(r-1)\sum_{i=1}^r\alpha_i^2-2\sum_{1\leq i<j\leq r}\alpha_i\alpha_j+r \sum_{i=1}^r(r+1-2i)\alpha_i\, .
        \end{equation*}
        
        \item For the operator $\dd_3$ we have
        \[ f^{(3)}_{\sS^\alpha V}(r)=\frac{\dot\delta_r^{(3)}(\alpha)}{(r-2)(r-1)(r+1)(r+2)}\frac{r_\alpha}{r}, \]
        where
        \begin{align*}
        \dot\delta_r^{(3)}(\alpha)= & 
        2(r-2)(r-1)\sum_{i=1}^r\alpha_i^3-6(r-2)\sum_{1\leq i\neq j\leq r}\alpha_i^2\alpha_j+24\sum_{1\leq i<j<k\leq r}\alpha_i\alpha_j\alpha_k \\ 
        + & 3r(r-2) \sum_{i=1}^r (r+1-2i)\alpha_i^2 -12r\sum_{i<j}(r+1-i-j)\alpha_i\alpha_j \\
        + & r^2\sum_{i=1}^r (6i^2-6i(r+1)+r^2+3r+2)\alpha_i\, . 
        \end{align*}
    \end{enumerate}
\end{thm}

Before the proof of Theorem~\ref{thm:casimiri} in Section~\ref{section:proof of Casimir}, let us recall in the next section some facts about universal enveloping algebras and Casimir elements.

\subsection{Universal enveloping algebra and Casimir elements}

We refer to \cite{Humphreys} for the background about Lie algebras and their universal enveloping algebras.
\medskip

If $\mathfrak{g}$ is a Lie algebra, we denote by $\ue(\mathfrak{g})$ its universal enveloping algebra.
Recall that $\ue(\mathfrak{g})$ is constructed by quotienting the tensor algebra $T(\mathfrak{g})$ via the bracket relations. This endows $\ue(\mathfrak{g})$ both with a natural filtration
\[ \ue_0(\mathfrak{g})\subset \ue_1(\mathfrak{g})\subset \ue_2(\mathfrak{g})\subset\cdots \]
and with a natural structure of Hopf algebra,
\[ \epsilon\colon \ue(\mathfrak{g})\longrightarrow\K\qquad\mbox{and}\qquad \Delta\colon \ue(\mathfrak{g})\longrightarrow \ue(\mathfrak{g})\otimes \ue(\mathfrak{g}). \]
Elements in $\ue_m(\mathfrak{g})/\ue_{m-1}(\mathfrak{g})$ correspond, via the Poincaré-Birkhoff–Witt Theorem, to symmetric $m$-tensors, and the maps $\epsilon$ and $\Delta$ are the counit and coproduct maps, respectively. 
The coproduct $\Delta$ describes the action on the tensor product of representations. 
\medskip

Our first remark is that the trace of the action of certain elements of $\ue(\ssl_r)$ satisfies additivity and log-multiplicativity identities. 

If $C\in\ue(\gl_r)$ and $W$ is a $\gl_r$-representation, then we denote by $\Tr_W C$ the trace of the action on $C$ on $W$.

\begin{lemma}\label{lemma:additivity for phi}
Let $C \in \ue(\ssl_r) \subset \ue(\gl_r)$ be an element such that $C\in\ker(\epsilon)\cap \ue_3(\ssl_r)$. 

For every $U_1,U_2 \in \Rep_\Qset$, the trace of the action of $C$ satisfies
$$\Tr_{U_1 + U_2} C = \Tr_{U_1} C + \Tr_{U_2} C \qquad\mbox{and}\qquad \Tr_{U_1 \cdot U_2} C = \Tr_{U_1} C \cdot \dim(U_2) + \dim(U_1) \cdot \Tr_{U_2} C.$$ 
\end{lemma}
\begin{proof}
Additivity always holds. 
Let us then look at log-multiplicativity.

First of all, let us notice that if $C$ has degree $1$, i.e.\ $C\in\ssl_r\subset\ue(\ssl_r)$, then it is primitive. Primitive here means that $\Delta(C)=C\otimes 1+1\otimes C$, so that the claim is straightforward.

Let us now assume that $C\in\ker(\epsilon)\cap \ue_2(\ssl_r)$ is quadratic. Let $U_1, U_2$ be polynomial representations of $\GL_r$. If $C$ is of the form $C = \sum_i x_i y_i$, where $x_i, y_i \in \ssl_r$, then
$$\Delta(C) = C \otimes 1 + 1 \otimes C + \sum_i (x_i \otimes y_i + y_i \otimes x_i).$$
As a consequence, the trace of the action of $C$ on $U_1 \otimes U_2$, which is prescribed by $\Delta(C)$, equals the trace of $C\otimes 1 + 1 \otimes C$, as $\Tr(x_i) = \Tr(y_i) = 0$ since every element in $\ssl_r$ is a commutator. However $\Tr_{U_1 \otimes U_2}(C \otimes 1) = (\Tr_{U_1} C) \cdot \dim(U_2)$ and similarly $\Tr_{U_1 \otimes U_2}(1 \otimes C) = \dim(U_1)\cdot(\Tr_{U_2} C)$. We thus obtain log-multiplicativity.

A similar proof works for cubic elements.
\end{proof}

\begin{remark}\label{rmk:se primitivo di grado 1 anche gl va bene}
    If $C$ has degree $1$ and is primitive, then the proof above works also if $C\in\gl_r$.
\end{remark}

\begin{remark}
    The proof above does not apply to elements of degree bigger or equal to four, the reason being similar as in Remark~\ref{rmk:più di tre non si va}.
\end{remark}

\begin{remark}
    We stated Lemma~\ref{lemma:additivity for phi} for elements in $\ue(\ssl_r)\subset\ue(\gl_r)$, but in fact only elements in the center were needed. In fact let us notice that the trace of the corresponding action is invariant under the Adjoint action of $\GL_r$. Therefore we can substitute such an element with its $\operatorname{U}_r$-mean, which is an element in the center of $\ue(\gl_r)$. 
\end{remark}

Finally, in this paper by \emph{Casimir element} of Lie algebra $\mathfrak{g}$ we mean every element in the centre of its universal enveloping algebra. Let us describe two special Casimir elements that will be useful in the next section.
\medskip

First of all, let us denote by $E_{ij}$ the elementary matrix whose only non-zero entry is $1$ in position $(i,j)$. These form the standard basis of $\gl_r$. 

Consider the following three elements of $\ue(\gl_r)$:
$$I = \sum_{i=1}^r E_{ii}, \qquad C_{2}=\sum_{i, j = 1}^r E_{ij}E_{ji}\qquad \mbox{and} \qquad C_{3} = \sum_{i, j, k=1}^r E_{ij}E_{jk}E_{ki}.$$

It is immediate to show that $I$, $C_2$ and $C_3$ belong to the center of $\ue(\gl_r)$, and are then Casimir elements (see for example \cite[Example~4.2.5]{Molev}).

We also consider the correction $C^{(2)} = \sum_{i, j = 1}^r (E_{ij}- \delta_{ij}I/r)(E_{ji} - \delta_{ji}I/r) \in \ue(\ssl_r)\subset \ue(\gl_r)$. Then $C^{(2)}$ still lies in the center of $\ue(\gl_r)$ and one easily computes
\begin{equation}\label{eqn:casimiro quadratico in sl}
C^{(2)} = C_2 - \frac{1}{r} I^2.
\end{equation}
Similarly, if $C_3^0 = \sum_{i, j, k=1}^r (E_{ij} - \delta_{ij}I/r)(E_{jk} - \delta_{jk}I/r)(E_{ki} - \delta_{ki}I/r)$ then
\begin{equation*}\label{eqn:casimiro cubico}
C_3^0 = C_3 - \frac{3}{r} C_2 I + \frac{2}{r^2} I^3\in \ue(\ssl_r)\subset \ue(\gl_r)\,.
\end{equation*}

Finally, let us set
\begin{equation}\label{eqn:casimiro cubico sl}
    C^{(3)}=2C_3^0-rC^{(2)}
\end{equation}

\subsection{Proof of Theorem~\ref{thm:casimiri}}\label{section:proof of Casimir}

The idea of the proof is to identify the coefficients $f^{(k)}_{\sS^\alpha V}(r)$ with the actions of certain Casimir elements. 

\begin{definition}\label{defn:phi k}
    Let $C^{(2)}$ and $C^{(3)}$ be the quadratic and cubic Casimir elements of $\ue(\ssl_r) \subset \ue(\gl_r)$ as defined in (\ref{eqn:casimiro quadratico in sl}) and (\ref{eqn:casimiro cubico sl}). 

    Define then
    \[ \phi^{(2)}_U(r) = \Tr_U \left(\frac{C^{(2)}}{(r-1)(r+1)}\right)\qquad\mbox{and}\qquad \phi^{(3)}_U(r)=\Tr_U\left(\frac{r\,C^{(3)}}{(r-2)(r-1)(r+1)(r+2)}\right). \]
\end{definition}

The main result of the section is the following proposition.

\begin{proposition}\label{prop:f = phi}
    For every $U\in\Rep^d_\Qset$ and $k=2,3$,
$$f^{(k)}_U(r)=\phi^{(k)}_{U}(r)\,.$$
\end{proposition}

Before the proof of Proposition~\ref{prop:f = phi}, we need the following remark. 

\begin{lemma}\label{lemma:Phi=f on Sym}
For $k=2,3$, we have the equalities 
\[ f^{(k)}_{S^m V}(r)=\phi^{(k)}_{S^m V}(r) \] 
for all $m \geq 0$.
\end{lemma}
\begin{proof}
    Recall that the $C^{(k)}$'s are constructed starting from the standard Casimir elements $I$, $C_2$ and $C_3$. One can explicitly compute the actions of the latter on the symmetric representation $S^m V$:
    \[ I|_{S^m V} = m \cdot \id, \qquad C_2|_{S^m V} = m (m+r-1) \cdot \id \qquad\mbox{and}\qquad C_3|_{S^m V} = m (m+r-1)^2\cdot \id\, . \]
    By definition it then follows that $C^{(2)}=C_2 - \frac{1}{r} I^2$ acts as
    \[ C^{(2)}|_{S^m V}=m (m+r)\frac{r-1}{r}\cdot\id \]
    and hence 
    \[ \phi^{(2)}_{S^m V}(r)=\frac{m (m+r)}{r+1}\frac{r_m}{r}=f^{(2)}_{S^m V}(r), \]
    where $r_m=\dim S^m V$ and the last equality is Example~\ref{example:symmetric}. 

    Continuing with the degree 3 case, recall that $C_3^0 = C_3 - \frac{3}{r} C_2 I + \frac{2}{r^2} I^3$, so that its action on $S^m V$ is
    \[ C_3^0|_{S^m V} = m (m+r)\frac{(r-1)}{r^2}[(r-2)m + r(r-1)]\cdot \id \]
    and for $C^{(3)}=2C_3^0-rC^{(2)}$ we get
    \[ C^{(3)}|_{S^m V} = \frac{m (m+r) (2m+r) (r-1) (r-2)}{r^2}\cdot \id. \]
    Finally, 
    \[ \phi^{(3)}_{S^m V}(r)=\frac{m (m+r) (2m+r)}{(r+1) (r+2)}\frac{r_m}{r}= f^{(3)}_{S^m V}(r), \]
    where again the last equality follows from Example~\ref{example:symmetric}.
\end{proof}

\proof[Proof of Proposition~\ref{prop:f = phi}]
By Lemma~\ref{lemma:additivity for phi} we know that, for $k=2,3$, $\phi^{(k)}_U(r)$ is additive and log-multiplicative. Similarly, by Proposition~\ref{prop:properties of fU}, $f^{(k)}_U(r)$ is weak additive and log-multiplicative.
Moreover, by Lemma~\ref{lemma:Phi=f on Sym} they agree on symmetric representations $S^m V$.

It follows that $\phi^{(k)}_U(r)$ and $f^{(k)}_U(r)$ also agree when $U=\Sym^\alpha V=S^{\alpha_1}V\cdots S^{\alpha_t}V$ is a monomial of symmetric representation, i.e.\
\[  f^{(k)}_{\Sym^\alpha V}(r)=\phi^{(k)}_{\Sym^\alpha V}(r). \]
Finally, it is enough to recall that the vector space $\Rep^d_\Qset$ is $\Qset$-linearly generated by the monomials $\Sym^\alpha V$ with $|\alpha|=d$ to conclude the proof.
\endproof

\begin{remark}
One important observation to make is the following: the maps $\phi^{(k)}$ are fully additive, whereas the operators $\dd_k$ are only weakly additive. The fact that $\dd_k(U)=\phi^{(k)}_U(r)\dd_2(V)$ seems to hint at the fact that they coincide on the whole $\Rep_\Qset$. However, this is false: as we already pointed out (cf.\ Example~\ref{example:counterexample to additivity}) the operators $\dd_k$ are not fully additive. 
\end{remark}
\medskip

Now, let us denote by $\mathsf{Z}(\ssl_r)$ the centre of $\ue(\ssl_r)$. By the Harish-Chandra Theorem, there is an isomorphism (see \cite[isomorphism~(4.15)]{Molev})
\[ \mathcal{H}\colon \mathsf{Z}(\gl_r)\stackrel{\sim}{\longrightarrow}\K[x_1,\dots,x_r]^{\mathfrak{S}_r} \]
where the variable $x_i$ corresponds to the central element $E_{ii}-i+1$.

\begin{lemma}\label{lemma:H action}
If $q(x)\in\K[x_1,\dots,x_r]^{\mathfrak{S}_r}$ is a symmetric polynomial, then for every Schur module we have
\begin{equation*} 
\mathcal{H}^{-1}(q(x))|_{\sS^\alpha V}=\dot{q}(\alpha)\cdot\id_{\sS^\alpha V}, 
\end{equation*}
where $\dot{q}(x)=q(x_1,x_2+1,\dots,x_r+r-1)$.
\end{lemma}
\begin{proof}
    Since $\mathcal{H}^{-1}(q(x))$ is central in $\ue(\ssl_r)$, it acts as a scalar multiple of the identity on any Schur module. The fact that the action is as described follows directly from the definition of the Harish--Chandra isomorphism.
\end{proof}

Let us then define the following symmetric polynomials: 
\begin{equation}\label{eqn:delta 2 dot}
    \delta_r^{(2)}(x)=(r-1)\sum_{i=1}^{r} x_i^2 -2\sum_{1\leq i<j\leq r} x_ix_j - \frac{r^2(r^2-1)}{12}
\end{equation}
and 
\begin{equation}\label{eqn:delta 3 dot}
    \delta_r^{(3)}(x)=2(r-2)(r-1)\sum_{i=1}^{r}x_i^3-6(r-2)\sum_{1\leq i\neq j\leq r} x_i^2x_j + 24\sum_{1\leq i<j<k\leq r}x_ix_jx_k\,.
\end{equation}

\begin{lemma}\label{lemma: deltas dot}
    After the change of variables $\alpha_i=x_i+i-1$, we have that 
    \begin{equation*}
        \dot\delta_r^{(2)}(\alpha)=(r-1)\sum_{i=1}^r\alpha_i^2-2\sum_{1\leq i<j\leq r}\alpha_i\alpha_j+r \sum_{i=1}^r(r+1-2i)\alpha_i\, .
    \end{equation*}
    and 
    \begin{align*}
        \dot\delta_r^{(3)}(\alpha)= & 
        2(r-2)(r-1)\sum_{i=1}^r\alpha_i^3-6(r-2)\sum_{1\leq i\neq j\leq r}\alpha_i^2\alpha_j+24\sum_{1\leq i<j<k\leq r}\alpha_i\alpha_j\alpha_k \\ 
        + & 3r(r-2) \sum_{i=1}^r (r+1-2i)\alpha_i^2 -12r\sum_{i<j}(r+1-i-j)\alpha_i\alpha_j \\
        + & r^2\sum_{i=1}^r (6i^2-6i(r+1)+r^2+3r+2)\alpha_i\, . 
    \end{align*}
\end{lemma}
\begin{proof}
    This is a direct substitution of variables.
\end{proof}

Notice that $\dot\delta_r^{(2)}$ and $\dot\delta_r^{(3)}$ coincide with the polynomials defined in Theorem~\ref{thm:casimiri}.

\begin{lemma}\label{lemma:i delta vanno bene}
    For $k=2,3$, let $\delta_r^{(k)}$ be as above. Then $\mathcal{H}^{-1}(\delta_r^{(k)})\in\ker(\epsilon)\cap\ue(\ssl_r)$.
\end{lemma}
\begin{proof}
    First of all, let us notice that $\mathcal{H}^{-1}(\delta_r^{(k)}(x))=\delta_r^{(k)}(E_{11},E_{22}+1,\dots,E_{rr}+r-1)$. By Lemma~\ref{lemma: deltas dot}, this polynomial has no degree $0$ term, so that by definition it belongs to $\ker(\epsilon)$.

    Let us now prove that $\mathcal{H}^{-1}(\delta_r^{(k)})\in\mathsf{Z}(\ssl_r)$.
    We first claim that for any constant $a$, the polynomial $\delta_r^{(k)}(x)$ is invariant under translation by $a$, i.e.\ 
    \[ \delta_r^{(k)}(x_1-a,\dots,x_r-a)=\delta_r^{(k)}(x_1,\dots,x_r). \] 
    This follows from an easy computation. 
    
    If we now put $a=\frac{1}{r}(x_1+\cdots x_r)$, then we find that $\delta_r^{(k)}(x)$ is in fact a polynomial in the variables $x_i-x_{i+1}$, i.e.\ there exists $\xi^{(k)}_r(y)\in\K[y_1,\dots,y_{r-1}]$ such that $\delta_r^{(k)}(x)=\xi^{(k)}_r(x_1-x_2,\dots,x_{r-1}-x_r)$. 
    Therefore $\mathcal{H}^{-1}(\delta_r^{(k)}(x))=\xi^{(k)}_r(E_{11}-E_{22},\dots,E_{r-1 r-1}-E_{rr})\in\mathsf{Z}(\ssl_r)$, thus concluding the proof.
\end{proof}

The following result will complete the proof of Theorem~\ref{thm:casimiri}.

\begin{proposition}\label{prop:ultima}
    Let $\phi^{(k)}$ be the characters defined in Definition~\ref{defn:phi k}, and let $\delta_r^{(k)}$ be the quadratic and cubic symmetric polynomials defined in (\ref{eqn:delta 2 dot}) and (\ref{eqn:delta 3 dot}). Then for every Schur module $\sS^\alpha V$ we have
    \[ \phi^{(2)}_{\sS^\alpha V}(r)=\frac{\dot\delta_r^{(2)}(\alpha)}{(r-1)(r+1)}\frac{r_\alpha}{r}\qquad\mbox{and}\qquad \phi^{(3)}_{\sS^\alpha V}(r)=\frac{\dot\delta_r^{(3)}(\alpha)}{(r-2)(r-1)(r+1)(r+2)}\frac{r_\alpha}{r}. \]
\end{proposition}
\begin{proof}
    First of all, by the proof of Lemma~\ref{lemma:Phi=f on Sym} and by a direct check using Lemma~\ref{lemma: deltas dot}, it follows that both equalities hold for symmetric representations $S^m V$.

    Now, by Lemma~\ref{lemma:additivity for phi} (which can be applied because of Lemma~\ref{lemma:i delta vanno bene}), both $\phi^{(k)}_{\sS^\alpha V}(r)$ and the functions $\frac{\dot\delta_2(\alpha)}{(r-1)(r+1)}\frac{r_\alpha}{r}$ and $\frac{\dot\delta_3(\alpha)}{(r-2)(r-1)(r+1)(r+2)}\frac{r_\alpha}{r}$ are additive and log-multiplicative. Therefore they agree on $\Rep_\Qset$ if and only if they agree on a set of generators. Since symmetric representations generate $\Rep_\Qset$, the claim follows.
\end{proof}

\proof[Proof of Theorem~\ref{thm:casimiri}]
The proof follows at once from Proposition~\ref{prop:f = phi} and Proposition~\ref{prop:ultima}.
\endproof

\begin{remark}
    Let us consider the central element $I=\sum E_{ii}\in\ue(\gl_r)$ of degree $1$. It is not an element in $\ue(\ssl_r)$, but by Remark~\ref{rmk:se primitivo di grado 1 anche gl va bene} the proof of Lemma~\ref{lemma:additivity for phi} continues to work for it. Therefore if we define 
    \[ \phi_U^{(1)}(r)=\Tr_U\left( \frac{1}{r} I \right) ,\]
    then $\phi^{(1)}(r)$ is additive and log-multiplicative. Explicitly we have 
    \[ \phi_{\sS^\alpha V}^{(1)}(r)=|\alpha|\frac{r_\alpha}{r}. \]
    In this way we understand Proposition~\ref{prop:slope} as a Casimir action. In fact, it is this observation that led us to think at the characters $f^{(k)}(r)$ as Casimir actions for $k=2$ and $3$.
\end{remark}

\begin{remark}
    The function 
     \[ f^{(2)}_{S^m V}(r)\cdot\left(\frac{r_m}{r}\right)=\frac{m(m+r)}{(r+1)}\left(\frac{r_m}{r}\right)^2 \] 
     counts the dimension of the representation $\sS^{m-1, m-1}V_{r+2}$, as it can be directly seen from Weyl's formula. Here $V_{r+2}$ is a vector space of dimension $r+2$. 
     
     Equivalently, it counts the number of standard tableaux of size $(m-1,m-1)$ for $\textrm{Mat}(2,r+2)$. These form a basis for the degree $m-1$ component of the homogeneous coordinate ring of the Grassmannian $\Gr(2,r+2)$ (see \cite[Lemma~8.14, Remark~8.16]{Mukai}), i.e.
    \[
    f(m)=h^0(\Gr(2,r+2), \mO(m-1)).    
    \]

    We point out that $f^{(2)}_{S^m V}(r)$ has another interesting Lie-theoretical connection when $r=4$, being the dimension of $V^{(k)}$, the distinguished module in the Severi variety with $a=4$ in \cite[Theorem~7.3]{lm06}. We have not explored a connection between these two facts yet.
    
    On the other hand, for $f^{(3)}_{S^m V}(r)$ we haven't found an interpretation valid for every $r$. However, for $r=3$, we can see that  $\frac{10}{3}\cdot f^{(3)}_{S^m V}(3)$ is the dimension of the Cartan power $\mathfrak{g}^{(m)}$ for $\mathfrak{g}=\mathfrak{so}_8$, see \cite[Theorem~7.1]{lm06}. 
    
    We do not have an explanation for such phenomenon, but a natural question is to understand if the polynomials appearing in Theorem~\ref{thm:casimiri} admit in general a geometrical interpretation.
\end{remark}

\section{Geometric applications}

\subsection{Schur functors of locally free sheaves}

In this section $X$ is a compact complex manifold (or a proper smooth algebraic variety over $\K$).

Let $E$ be a vector bundle of rank $r$ on $X$. Define the following characteristic classes:
\begin{align*}
    \Delta_1(E)= & \cc_1(E) \\
    \Delta_2(E)= & \cc_1(E)^2-2r\ch_2(E) \\
    \Delta_3(E)= & \cc_1(E)^3-3r\cc_1(E)\ch_2(E)+3r^2\ch_3(E).
\end{align*}

\begin{remark}
    These classes are the same as defined in (\ref{eqn:Delta 1 2 3}). We refer to \cite{DrezetLog,Ful20} for other appearances of logarithmic Chern classes.
\end{remark}

For any partition $\alpha=(\alpha_1,\dots,\alpha_r)$, we denote by $\sS^\alpha E$ the \emph{Schur bundle} associated to $E$. It is the vector bundle of rank $r_\alpha$ such that 
\[ (\sS^\alpha E)_x=\sS^\alpha E_x\qquad\forall x\in X. \]

As a consequence of the results in the previous section, we get the following statement.

\begin{thm}\label{thm:fibrati}
    Let $X$ and $E$ be as above. Then 
    \begin{align*}
        \Delta_1(\sS^\alpha E) = & \, |\alpha|\frac{r_\alpha}{r} \Delta_1(E) \\
        \Delta_2(\sS^\alpha E) = & \, \frac{\dot\delta_r^{(2)}(\alpha)}{(r-1) (r+1)} \left(\frac{r_\alpha}{r}\right)^2 \Delta_2(E) \\
        \Delta_3(\sS^\alpha E) = & \, \frac{\dot\delta_r^{(3)}(\alpha)}{(r-2) (r-1) (r+1) (r+2)} \left(\frac{r_\alpha}{r}\right)^3 \Delta_2(E)\,,
    \end{align*}
    where $\dot\delta_r^{(2)}(\alpha)$ and $\dot\delta_r^{(3)}(\alpha)$ are as in Theorem~\ref{thm:casimiri}.
\end{thm}

\begin{proof}
    Let us assume that $E=L_1\oplus\cdots\oplus L_r$, where the $L_i$'s are line bundle with $\cc(L_i)=a_i$, i.e.\ the $a_i$'s are the Chern roots of $E$. Then by definition
    \[ \ch(E)=\sum_{i=1}^{r}e^{a_i}, \]
    which is the same expression as in Definition~\ref{defn:ch for rep}.
    In particular every expression involving the Chern character of $E$ and $\sS^\alpha E$ is equivalent to the corresponding expression for the vector spaces $E_x$ and $\sS^\alpha E_x$, for any point $x\in X$.

    The claim then follows at once from Proposition~\ref{prop:slope}, Theorem~\ref{thm:Delta 2 e 3} and Theorem~\ref{thm:casimiri}.

    If $E$ is any vector bundle, then by the splitting principle we can formally reduce to the case above, thus concluding the proof.
\end{proof}

\begin{remark}
    The case $\Delta_1=\cc_1$ was already proved in \cite{Rub13}. Rubei's proof uses a double-induction argument to reduce to the symmetric case, where the result is easy to prove. Also our proof is deduced from the symmetric case, but our reduction argument is not inductive and is instead based on representation-theoretic arguments. 
\end{remark}

\begin{remark}
    The expressions in Theorem~\ref{thm:fibrati} can be used to deduce closed expressions for the first three Chern characters of a Schur functor, see Section~\ref{appendix:Schur}. Similar expressions have been investigated in \cite[Section~4]{Isopoussu}. More precisely, in \cite[Theorem~4.3]{Isopoussu} there is an asymptotic version of the expression for $\ch_2(\sS^\alpha E)$ under some (strong) assumptions on the partition $\alpha$.
\end{remark}

\begin{remark}
    It is interesting to notice that for any partition $\alpha$, the polynomial $\dot\delta_r^{(2)}(\alpha)$ is non-negative. It follows that if $X$ is a K\"ahler manifold of dimension $N$, then 
    \[ \Delta_2(\sS^\alpha E).\omega^{N-2}\geq0\qquad\iff\qquad \Delta_2(E).\omega^{N-2}\geq0. \] 
    Here $\omega$ is a K\"ahler class on $X$. 
\end{remark}

The comment above should be read with an eye on the Bogomolov inequality, which says that if a vector bundle $E$ is slope $\omega$-semistable on a compact K\"ahler manifold, then $\Delta_2(E).\omega^{N-2}\geq0$. The vice versa is not true, but it is known that we can still deduce polystability of $\sS^\alpha E$ from that of $E$. We state here the following known result for lack of a better reference.

\begin{proposition}\label{prop:poly implica poly}
    Let $X$ be a compact K\"ahler manifold, $\omega$ a K\"ahler class and $E$ a slope $\omega$-polystable vector bundle on $X$. Then for any partition $\alpha$ the Schur functor $\sS^\alpha E$ is slope $\omega$-polystable.
\end{proposition}
\begin{proof}
    If we put $\ell=|\alpha|$, then $\sS^\alpha E\subset E^{\otimes\ell}$ as a direct summand. It follows from the Donaldson--Uhlenbeck--Yau Theorem (see for example \cite{LT95}) that $E^{\otimes\ell}$ is slope $\omega$-polystable, so that $\sS^\alpha E$ is also $\omega$-polystable.
\end{proof}

\begin{remark}\label{remark:not stable}
    In Proposition~\ref{prop:poly implica poly}, even if $E$ is slope stable, the polystability of $\sS^\alpha E$ cannot in general be improved. As an example, consider $E= S^2 Q$, where $Q$ is the quotient tautological bundle on a grassmannian $\operatorname{Gr}(k,n)$. Then $E$ is slope stable, but $S^2 E=S^2(S^2 Q))$ is not indecomposable, hence it is strictly polystable. 
    On the other hand, indecomposability of $\sS^\alpha E$ is the only obstruction to stability: if $\sS^{\alpha} E$ is indecomposable, then it is stable.
\end{remark}

\subsection{Modular bundles on hyper-K\"ahler manifolds}

Let $X$ be a compact hyper-K\"ahler manifold. Recall that the second integral cohomology group $\oH^2(X,\mathbb{Q})$ is torsion free and endowed with a non-degenerate quadratic form $\mathsf{q}_X$. With an abuse of notation, we keep denoting by $\mathsf{q}_X$ 
the class in $\oH^4(X,\mathbb{Q})$ dual to this quadratic form.

We denote by $\operatorname{SH}(X)\subset\oH(X,\mathbb{Q})$ the Verbitsky component, namely the submodule generated by $\oH^2(X,\mathbb{Q})$ via cup product. Notice that $\mathsf{q}_X\in\operatorname{SH}^4(X)$.

\begin{definition}[\protect{\cite[Definition~1.1]{O'Grady:Modular}}]\label{defn:modular}
    A vector bundle $E$ on $X$ is called \emph{modular} if the projection of $\Delta_2(E)$ onto $\operatorname{SH}^4(X)$ is a multiple of $\mathsf{q}_X$.
\end{definition}

As a corollary of Theorem~\ref{thm:fibrati} we obtain the following result that allows to construct new modular bundles on hyper-K\"ahler manifolds from old ones.

\begin{thm}\label{thm:modular}
    Let $X$ be a compact hyper-K\"ahler manifold. If $E$ is a slope polystable modular vector bundle, then $\sS^\alpha E$ is a slope polystable modular vector bundle.
\end{thm}

\begin{proof}
   The slope polystability follows from Proposition~\ref{prop:poly implica poly}.  The modularity follows from Theorem~\ref{thm:fibrati}: in fact the projection into $\operatorname{SH}^4(X)$ of $\Delta_2(\sS^\alpha E)$ is a multiple of the projection of $\Delta_2(E)$, which is a multiple of $\mathsf{q}_X$ by hypothesis.
\end{proof}

\begin{example}
    Let $S$ be a K3 surface, i.e.\ a hyper-K\"ahler surface, and let $E$ be a stable vector bundle with Mukai vector $v(E)=(r,\cc,s)$, where $r$ is the rank of $E$. By dimension reasons, the bundle $E$ is modular.

    For the same dimension reasons, any Schur functor $\sS^\alpha E$ is also modular. The Mukai vector of $\sS^\alpha E$ is
    \[ v(\sS^\alpha E)=\left( r_\alpha , |\alpha|\frac{r_\alpha}{r}\cc , \frac{(|\alpha|^2-\tilde{\delta}_r^{(2)}(\alpha))}{r}\frac{r_\alpha}{r}\frac{\cc^2}{2}+\tilde{\delta}_r^{(2)}(\alpha)(s-r)\frac{r_\alpha}{r}+r_\alpha \right), \]
    where we have used the same notations as in Section~\ref{appendix:Schur}, so that $\tilde{\delta}_r^{(2)}(\alpha)=\frac{\dot\delta_r^{(2)}(\alpha)}{(r-1)(r+1)}$.
    
    The Mukai vector $v(\sS^\alpha E)\in\oH^*(S,\Zset)$ can be primitive or not depending on $\alpha$, $r$ and the bundle $E$ itself. For example, let us suppose that $\operatorname{Pic}(S)=\Zset\cdot H$, where $H$ is ample and $H^2=2d\geq 6$. Then there exists a $H$-stable vector bundle $E$ of rank $2$ with $\cc_1(E)=H$ and $\ch_2(E)=0$. Such a vector bundle has a primitive Mukai vector $v(E)=(2,H,2)$. The symmetric power $S^2 E$ has Mukai vector
    \[ v(S^2 E)=(3,3H,d+3), \]
    therefore it is primitive for all $d\nmid3$, but it is not primitive for $d\in 3\Zset$.
\end{example}

\begin{example}
    Examples of stable modular vector bundles on hyper-K\"ahler fourfolds of type $\operatorname{K3}^{[2]}$ have been constructed in \cite{O'Grady:Modular with many moduli,Bottini,Fa24,FO24,FT25}. By Theorem~\ref{thm:modular}, each of these examples gives rise to many more (polystable) examples.
\end{example}

\appendix

\section{Discriminants for low-rank bundles}\label{appendix:low ranks}

Let us consider the logarithmic classes defined in (\ref{eqn:Delta 1 2 3}). As already remarked in the introduction, these logarithmic classes $\Delta_k(F)$ for a coherent sheaf $F$ are linked to the the stability of $F$ itself. This is well known for $k=1,2$, and there are evidences starting to build up for $k=3$, see \cite{Ful20}.
It is therefore interesting to study not only the first logarithmic classes, but also the higher-degree terms.  
As a first observation, notice that \emph{all} logarithmic classes behave well with respect to twist with line bundles: for example (see \cite[Proposition 2.1]{DrezetLog}) we have for $E$ a vector bundle and $L$ a line bundle
\[
\Delta_k(E \otimes L)= \Delta_k(E), \ k \geq 2
\]

The first three logarithmic classes satisfies another nice property.
\begin{proposition}
    Let $E$ be a locally free sheaf of rank $ r< k$, for $k=1,2,3$. Then $\Delta_k(E)=0$.
\end{proposition}
\begin{proof}
    For $k=1$ the claim is trivially true. For $k=2, 3$ this is not evident from the definition. However, if we write $\Delta_k$ in terms of the Chern classes, we find that
\[     \Delta_2(E)=-(r-1)\cc_1^2(E)+2r\cc_2(E)     \] 
and
\[     \Delta_3(E)=\frac{1}{2} \left( (r-1)(r-2)\cc_1(E)^3-3 r (r-2)\cc_1(E)\cc_2(E)+3r^2\cc_3(E)\right).     \] The result then follows by simply noting that $\cc_i(E)=0$ for $i > r$.
\end{proof}

The above result does not hold for $k \geq 4$, highlighting another difference between lower and higher degree discriminants. In fact, $\Delta_4(E)$ is already not zero on rank 2 bundles with $\cc_2 \neq 0$. However, it is possible to find suitable modifications for both $\Delta_4$ and $\Delta_5$ that behave well with respect to this property. For this, let us expand the degree $k=4,5$ the expression in Definition \ref{defn:Log chern}:
\begin{align*}
    \Delta_4(E)= & \, \ch_1(E)^4 - 4 r \ch_1(E)^2 \ch_2(E) + 2 r^2 (\ch_2(E)^2 + 2 \ch_1(E) \ch_3(E)) - 4 r^3 \ch_4(E) \\
   \Delta_5(E)= & \, \ch_1(E)^5 - 5 r \ch_1(E)^3 \ch_2(E) + 5 r^2 \ch_1(E) (\ch_2(E)^2 + \ch_1(E) \ch_3(E)) \\
   - & \, 5 r^3 (\ch_2(E) \ch_3(E) + \ch_1(E) \ch_4(E)) + 5 r^4 \ch_5(E) r^4\,.
\end{align*}

We have the following result.
\begin{proposition}
    Let $E$ be a locally free sheaf of rank $r$ and let us consider the two modified classes
    \begin{align*}
    \widetilde \Delta_4(E)= & \, (r+1)\Delta_4(E)-\Delta_2(E)^2 \\
    \widetilde \Delta_5(E)= & \,(r+5)\Delta_5(E)-5\Delta_2(E)\Delta_3(E).
    \end{align*}
    Then $\widetilde\Delta_k(E)=0$ for $r<k$ and $k=4,5$.
\end{proposition}

\begin{proof}
    Once again, it is enough to expand the two classes in term of the Chern classes. We have:
    \begin{align*}
  \widetilde \Delta_4= & \,  \frac{1}{6} r (-\cc_1^4 (-3 + r) (-2 + r) (-1 + r) + 
   4 \cc_1^2 \cc_2 (-3 + r) (-2 + r) r \\
   - & \, 2 \cc_2^2 (-3 + r) (-2 + r) r - 
   4 \cc_1 \cc_3 (-3 + r) r (1 + r) + 4 \cc_4 r^2 (1 + r))
   \end{align*}
   and
   \begin{align*}
    \widetilde \Delta_5= & \, \frac{1}{24} r (\cc_1^5 (-4 + r) (-3 + r) (-2 + r) (-1 + r) - 
   5 \cc_1^3 \cc_2 (-4 + r) (-3 + r) (-2 + r) r \\
   + & \, 5 \cc_1^2 \cc_3 (-4 + r) (-3 + r) r (2 + r) + 5 \cc_1 (-4 + r) r (\cc_2^2 (-3 + r) (-2 + r) - \cc_4 r (5 + r)) \\
   + & \, 5 r^2 (-\cc_2 \cc_3 (-4 + r) (-3 + r) + \cc_5 r (5 + r))).
    \end{align*}
    The result then follows.
\end{proof}

One can of course try to generalise this to higher discriminants, but we could not find any rigid pattern. It would be anyway very interesting to have an explicit expression for classes $\widetilde\Delta_k$, for any $k\geq1$, with the property that $\widetilde\Delta_k(E)=0$ for every locally free sheaf $E$ of rank $r<k$.

\section{Explicit formulas}\label{appendix}
In this appendix we collect some explicit formulas for the computation of Chern classes of Schur bundles. Sections~\ref{appendix:symmetric} and~\ref{appendix:exterior} follows from \cite{Dragutin}, while Section~\ref{appendix:Schur} is a consequence of Theorem~\ref{thm:fibrati}.
\medskip

From now on $E$ is a locally free sheaf of rank $r$, $\alpha=(\alpha_1,\dots,\alpha_r)$ a partition of size $|\alpha|=\alpha_1+\cdots+\alpha_r$ and $r_\alpha$ the rank of the Schur bundle $\sS^\alpha E$. 

Finally, for sake of simplicity, we use the shorthand
\[ \cc_1:=\cc_1(E)\qquad\mbox{and}\qquad\ch_i:=\ch_i(E). \]

\subsection{The symmetric bundle}\label{appendix:symmetric}
The following is an expansion of \cite[Theorem~4.8]{Dragutin}, see also (\ref{eqn:ch of Sym}).
\begin{align*}
    \rk(S^m E) = & \,\, r_m:=\binom{m+r-1}{r-1} \\
    \cc_1(S^m E)= & \,\, m\frac{r_m}{r} \cc_1 \\
    \ch_2(S^m E) = & \,\, \frac{1}{2}\frac{(m-1) m}{r+1}\frac{r_m}{r}\cc_1^2+\frac{m (m+r)}{r+1}\frac{r_m}{r}\ch_2 \\
    \ch_3(S^m E) = & \,\, \frac{1}{6}\frac{(m-2) (m-1) m}{(r+1) (r+2)}\frac{r_m}{r}\cc_1^3+\frac{(m-1) m (m+1)}{(r+1) (r+2)}\frac{r_m}{r}\cc_1\ch_2+\frac{m (m+r) (2m+r)}{(r+1) (r+2)}\frac{r_m}{r}\ch_3 \\
    \Delta_2(S^m E) = & \,\, \frac{m (m+r)}{r+1}\left(\frac{r_m}{r}\right)^2\Delta_2(E) \\
    \Delta_3(S^m E) = & \,\, \frac{m (m+r) (2m+r)}{(r+1) (r+2)}\left(\frac{r_m}{r}\right)^3\Delta_3(E)
\end{align*}
\vspace{2cm}
\subsection{The exterior bundle}\label{appendix:exterior}

The following is an expansion of \cite[Theorem~4.2]{Dragutin}.

\begin{align*}
    \rk(\W^n E) = & \,\, r_n:=\binom{r}{n} \\
    \cc_1(\W^n E) = & \,\ n\frac{r_n}{r}\cc_1 \\
    \ch_2(\W^n E) = & \,\, \frac{1}{2}\frac{(n-1) n}{r-1}\frac{r_n}{r}\cc_1^2+ \frac{n (r-n)}{r-1}\frac{r_n}{r}\ch_2 \\
    \ch_3(\W^n E) = & \,\, \frac{1}{6}\frac{(n-2) (n-1) n}{(r-2) (r-1)}\frac{r_n}{r}\cc_1^3 + \frac{(n-1) n (r-n)}{(r-2) (r-1)}\frac{r_n}{r}\cc_1\ch_2+ \frac{n (2n^2-3rn+r^2)}{(r-2) (r-1)}\frac{r_n}{r} \ch_3 \\
    \Delta_2(\W^n E) = & \,\, \frac{n (r-n)}{r-1}\frac{r_n}{r}\Delta_2(E) \\
    \Delta_3(\W^n E) = & \,\, \frac{n (2n^2-3nr+r^2)}{(r-2) (r-1)}\frac{r_n}{r}\Delta_3(E)
\end{align*}

\subsection{The Schur bundles}\label{appendix:Schur}

The functions $\dot\delta_2(\alpha,r)$ and $\dot\delta_3(\alpha,r)$ are the same as defined in Theorem~\ref{thm:casimiri}. For sake of readability, let us introduce this notation:
\[ \tilde{\delta}_r^{(2)}:=\frac{\dot\delta_r^{(2)}(\alpha)}{(r-1)(r+1)}\qquad\mbox{and}\qquad \tilde{\delta}_r^{(3)}:=\frac{\dot\delta_r^{(3)}(\alpha)}{(r-2)(r-1)(r+1)(r+2)}.  \]

\begin{align*}
    \rk(\sS^\alpha E) = & \,\, r_\alpha:=\prod_{1\leq i<j\leq r}\frac{\alpha_i-\alpha_j+j-i}{j-i} \\
    \cc_1(\sS^\alpha E) = & \,\, |\alpha|\frac{r_\alpha}{r}\cc_1 \\
    \ch_2(\sS^\alpha E) = & \,\, \frac{1}{2r}\left( |\alpha|^2-\tilde{\delta}_r^{(2)}\right)\frac{r_\alpha}{r}\cc_1^2 + \tilde{\delta}_r^{(2)}\frac{r_\alpha}{r}\ch_2 \\
    \ch_3(\sS^\alpha E) = & \,\, \frac{1}{6r^2}\left( |\alpha|^3-3|\alpha|\tilde{\delta}_r^{(2)}+2\tilde{\delta}_r^{(3)}\right)\frac{r_\alpha}{r}\cc_1^3+\frac{1}{r}\left(|\alpha|\tilde{\delta}_r^{(2)}-\tilde{\delta}_r^{(3)}\right)\frac{r_\alpha}{r}\cc_1\ch_2+\tilde{\delta}_r^{(3)}\frac{r_\alpha}{r}\ch_3 \\
    \Delta_2(\sS^\alpha E) = & \,\, \tilde{\delta}_r^{(2)}\left(\frac{r_\alpha}{r}\right)^2\Delta_2(E) \\
    \Delta_3(\sS^\alpha E) = & \,\, \tilde{\delta}_r^{(3)}\left(\frac{r_\alpha}{r}\right)^3\Delta_3(E)
\end{align*}

\bibliographystyle{alpha}

\end{document}